\documentclass[12pt]{amsart}
\usepackage{amsmath,amsthm}
\usepackage{amssymb}
\usepackage{amsfonts}
\usepackage{amscd}
\usepackage{booktabs}
\usepackage{slashbox}

\usepackage{color}

\usepackage{dsfont} 
\usepackage{mathtools} 
\usepackage{braket} 
\input{xy}
\xyoption{all}

\newtheorem{thm}{Theorem}[section]
\newtheorem{theorem}[thm]{Theorem}

\newtheorem{lemma}[thm]{Lemma}

\newtheorem{proposition}[thm]{Proposition}
\theoremstyle{definition}

\newtheorem{definition}[thm]{Definition}
\newtheorem{remark}[thm]{Remark}

\newcommand{\N} {\mathbf{N}}
\newcommand{\R} {\mathbb{R}}
\newcommand{\Z} {\mathbf{Z}}

\newcommand{\C} {\mathbb{C}}

\newcommand{\bP}{\mathbb{P}}

\newcommand{\cA}{\mathcal{A}}

\newcommand{\cF}{\mathcal{F}}

\newcommand{\cI}{\mathcal {I}}

\newcommand{\cM}{{\mathcal M}}
\newcommand{\cO}{\mathcal{O}}
\newcommand{\cP}{\mathcal{P}}

\newcommand{\cU}{\mathcal{U}}
\newcommand{\cL}{\mathcal{L}}
\newcommand{{\cHol}}{{\cal{{H}}}}
\newcommand{\cV}{\mathcal V}

\newcommand{\al}{\alpha}

\newcommand{\ep}{\epsilon}

\newcommand{\gk}{\frak g}
\newcommand{\hk}{\frak h}

\newcommand{\fg}{\mathfrak g}

\newcommand{\fn}{\mathfrak n}
\newcommand{\fh}{\mathfrak h}

\newcommand{\fb}{\mathfrak b}

\newcommand{\fp}{\mathfrak p}

\newcommand{\sets} {\mathrm{ (sets) }}

\newcommand{\smflds}{ \mathrm{(smflds) }}

\newcommand{\supgrp}{ \mathrm{(sgroups) }}

\newcommand{\GL}{\mathrm{GL}}

\newcommand{\rSL}{\mathrm{SL}}

\newcommand{\rEnd}{{\mathrm{End}}}

\newcommand{\Lie}{\mathrm{Lie}}

\newcommand{\id}{{\mathrm{id}}}
\newcommand{\rVec}{{\mathrm{Vec}}}
\newcommand{\op}{{\mathrm{op}}}

\newcommand{\red}[1]{\widetilde{ #1 }}
\newcommand{\wt}[1]{\widetilde{ #1 }}
\newcommand{\wh}[1]{\widehat{ #1 }}
\newcommand{\restr}[2]{{{#1}}_{{\left. \right|}_{#2}}}

\usepackage{color}

\newcommand{\lra} {\longrightarrow}

\newcommand{\di}{{\rm d}}

\newcommand{\beq}{\begin{equation}}
\newcommand{\eeq}{\end{equation}}

\begin{document}
\title{ Super Bundles} 

\author{Claudio Carmeli}
\address{DIME, 
Universit\`a di Genova, Genova, Italy}
\email{ carmeli@dime.unige.it}

\author{Rita Fioresi}
\address{Dipartimento di Matematica, Universit\`{a} di
Bologna, Piazza di Porta S. Donato, 5. 40126 Bologna, Italy.}
\email{rita.fioresi@unibo.it}

\author{V.S.Varadarajan}
\address{Department of Mathematics, UCLA, Los Angeles, Los Angeles, CA 90095-1555, USA}
\email{vsv@math.ucla.edu}

\begin{abstract}
In this paper we give a brief account of the main aspects of the theory of associated and principal super bundles. As an application, we review the  Borel-Weil-Bott Theorem in the super setting, and some results on projective embeddings of homogeneous spaces.
\end{abstract}

\maketitle

\section{Introduction}

In this paper we want to discuss the basic aspects of the theory 
of associated super vector bundles and principal super bundles over
supermanifolds, together with some applications.
We are interested in both the real differentiable and the 
complex analytic categories, so our ground field is $k=\R$ or $\C$.
In the end, we shall also make some remarks on 
the algebraic category.

\medskip
A Lie supergroup (SLG) is a group object 
in the category of supermanifolds $\smflds$ (real differentiable
or complex analytic).
Morphisms of Lie supergroups are morphisms of the underlying
supermanifolds preserving the group structure. We shall denote
the category of Lie supergroups with $\supgrp$.
We have three different and equivalent ways to view a Lie supergroup
(Refs. \cite{ccf} Ch. 7, \cite{koszul}, \cite{vi2}, \cite{cf}):

\begin{enumerate} 
\item As a supermanifold, that is as
pair $(\red{G}, \cO_G)$, where $\red{G}$ is a Lie group and 
$\cO_G$ a sheaf of superalgebras, with multiplication and inverse
morphisms;

\item As a group valued representable functor $G:\smflds \lra \sets$;

\item As a Super Harish-Chandra pair (SHCP), that is a pair
$(\red{G},\fg)$, where $\red{G}$ is a Lie group and $\fg$ a super
Lie algebra, with $\fg_0 \simeq {\rm Lie}(\red{G})$ together
with some natural compatibility conditions.
\end{enumerate}

The purpose of the present note is to show how to translate this equivalence,
when considering vector bundles or principal bundles on supermanifolds, which
carry a natural SLG action.

\medskip
The material we expose is generally known, however, given the several equivalent
approaches to the theory of supergroups, we think the reader can
benefit by seeing the various approaches to the theory of super vector
and principal bundles together with the equivalences properly spelled
out in detail.
Furthermore, we provide important applications, namely the Borel-Weyl-Bott
theorem and projective embeddings of supermanifolds, which
have an interest on their own.

\section{Super Bundles}

In this section we introduce various types of super bundles
and we prove the equivalence between several definitions. For
more details refer to \cite{bcc, bcf, ccf, vsv2} as well as the more
classical references \cite{dm, Leites, Kostant, ma}.

\subsection{Representations of Supergroups}

We start by defining the concept of linear
action of a SLG on a super vector space.

\begin{definition} \label{action-def}
Let $G$ be a SLG and $V$ a 
finite dimensional super vector space. 
We say that we have an \textit{action} of $G$ on $V$
if we have a natural transformation:
$$
G(\cdot) \times V(\cdot) \lra V(\cdot), \qquad g, v \mapsto g \cdot v
$$
satisfying the usual diagrams together with linearity, that is:
$$
g \cdot (u+v)=g \cdot u + g \cdot v, \qquad g \cdot \lambda u=
\lambda (g \cdot u), \quad g \in G(T), \, u,v \in V(T), \, \lambda \in \cO(T)_0
$$
with $V(\cdot)$ the functor:
$$
V(\cdot):\smflds \lra \sets, \qquad V(T)=(\cO(T) \otimes V)_0
$$
where $T=(\red{T},\cO_T) \in \smflds$, $\cO(T)$ the superalgebra
of global sections.
\end{definition}
We now establish the equivalence of this notion with others.
The following fact is a simple verification (see also
\cite{ccf} Ch. 7, 8, 9).

\begin{proposition}
\label{Proposition::equivrepr}
Let $G$ be a SLG , and $V=V_0\oplus V_1$ a 
finite-dimensional vector superspace. 
The following notions are equivalent.
\begin{enumerate}
\item Action of $G$ on $V$ according to Def. \ref{action-def}
\[
G(\cdot)\times V(\cdot) \rightarrow V(\cdot)
\]
We will refer to this as a 
{\sl $G$ linear action via the functor of points}.

\item A morphism of supermanifolds:
\[
a:\,G\times V \rightarrow V
\]
obeying the usual commutative diagrams and satisfying:
\[
a^\ast(V^\ast)\subseteq \cO(G)\otimes V^\ast
\]
We will refer to this as a {\sl $G$ linear action}.

\item SLG's morphism
\[
G\rightarrow \GL{(V)}
\]
We will refer to this as a {\sl $G$-representation}.


\item A natural transformation
\[
G(\cdot)\rightarrow \GL(V)(\cdot)
\]
We will refer to this as a 
{\sl $G$-representation via the functor of points}.


\item A SHCP representation, that is:
\begin{enumerate}
\item a Lie group morphism
\[
\tilde{\pi}:\, \red{G} \rightarrow \GL(V_0)\times \GL(V_1)
\]
\item a super Lie algebra morphism
\[
\rho^\pi:\,\gk \rightarrow \rEnd(V)
\]
such that
\begin{align*}
\tilde{\pi}(g)\rho^\pi(X) \tilde{\pi}(g)^{-1} =\rho^\pi(Ad(g)X), \qquad
\restr{\rho^\pi}{\fg_0} \simeq \di\tilde{\pi}
\end{align*}
\end{enumerate}

\end{enumerate}
\end{proposition}


\begin{remark}
Notice that the first four characterization of the concept of action
are merely an application of Yoneda's lemma; the only check 
concernes the equivalence between any of the first four notion with
the fifth one. 
The above proposition reflects, at the level of representation theory, 
the equivalence existing between SLGs, SHCPs and the functor of points picture.
\end{remark}

Let us now introduce 
the concept of {\sl contragredient
representation}.

\begin{definition}\label{contra-rep}
Let $\pi:G(\cdot) \lra \GL(V)(\cdot)$ be a $G$-representation. 
As in the classical setting, we have that $\pi$ induces another
representation on $V^*$ that we call the
\textit{contragredient representation}. Such a representation
is given by:
$$
\pi_c(g)(f)(v)=f(\pi(g^{-1})v), \qquad f \in V^*
$$
Equivalently if
$\pi=(\tilde{\pi},\rho^\pi)$ is a SHCP's representation of $(\red{G},\gk)$ 
on $V$, 
the \textit{contragredient representation} $(\tilde{\pi}_c,\rho^\pi_c)$
with respect to $(\tilde{\pi},\rho^\pi)$ is defined as:
\begin{align*}
\tilde{\pi}_c(g)(f)(v)  \coloneqq f(\tilde{\pi}(g^{-1})v),\qquad
\rho^\pi_c(X)(f)(v)  \coloneqq f(\rho^\pi(-X)v)
\end{align*}
with \(f\in V^\ast\,,\, v\in V\,,\, g\in \red{G}\,,\, X\in \fg\).
Given an action $a$ of $G$ in $V$, we 
shall denote the corresponding contragredient action with $a_c$. 
\end{definition}

\subsection{Super Vector Bundles and associated bundles}
We now want to define the concept of super vector bundle on $G/H$
associated to a finite dimensional
$H$-representation, where $H$ a closed subSLG of $G$.
Classically if $\red{\sigma}$ is a representation in $V$ of the ordinary
Lie group $\red{H}$ a closed subgroup of $\red{G}$, the global sections of the
associated bundle consist of the $\red{H}$-covariant functions, that is
the functions $f:G \lra V$ satisfying:
\begin{align}
\label{eq::clcov}
f(gh)=\red{\sigma}(h)^{-1}f(g)
\end{align}

We now want to give this same concept in supergeometry
in the three different settings, SLG's, SLG's
through the functor of points and SHCP's in the same spirit as in
Prop. \ref{Proposition::equivrepr}. Preliminary to this, let us recall
the concept of \textit{super vector bundle}
(see, for example, \cite{dm, ccf}). In the following, we let \(k=\R\,,\,\C\).

\begin{definition}
Let $M=(\red{M}, \cO_M)$ be a supermanifold. 
A \textit{super vector bundle} $\cV$ of rank $p|q$
is a locally free sheaf of rank $p|q$ 
that is for each $x \in \red{M}$ there exist $U$ open 
such that $\cV(U) \cong \cO_M(U)^{p|q}:= \cO_M(U) \otimes k^{p|q}$.
$\cV$ is a sheaf
of $\cO_M$ modules and at each $x \in \red{M}$, the stalk $\cV_x$ is
a $\cO_{M,x}$ module.
We define the \textit{fiber} of $\cV$ at
the point $x$ as the vector superspace $\cV_x/m_{x}\cV_x$, where
$m_x$ is the maximal ideal of $\cO_{M,x}$.  

\medskip

More explicitly,
if $\cV(U) \cong \cO_M(U)^{p|q}$, we have that the
stalk at $x$ is $\cV_x=\cO_{M,x}^{p|q}$, while the fiber is  $k^{p|q}$.
\end{definition}

\begin{definition}
Let $G$ be a SLG, $H$ a closed subSLG, $\sigma$ a 
finite-dimensional representation of $H$ on $V$, with
$\sigma=(\widetilde{\sigma}, \rho^\sigma)$ in the language of SHCP's. 
Consider the sheaf over $\red{G}/\red{H}$
\[
\cA(U) \coloneqq \cO_G (\red{p}^{-1}(U)) \otimes V
\]
where $p:G\rightarrow G/H$ is the canonical submersion.

\begin{itemize}
\item We define 
in the SLG context the assignment:
\begin{align}
\label{eq::assbundleSLG}
U\mapsto \cA_{SLG}(U)
\end{align}
where:
\begin{align} 
\label{eq::sheafSLG}
\cA_{SLG}(U) & \coloneqq\set{f\in \cA(U)\, |\, 
( \mu^\ast_{G,H}\otimes 1) (f) = (1\otimes a^\ast_c) f}
\end{align}
and
\[
\mu_{G,H}:\, G\times H \, \stackrel{\begin{array}{c}1\times i\\\end{array}}
{\hookrightarrow} \, G\times G \, \stackrel{
\begin{array}{c}\mu\\ \end{array}}{\rightarrow} \, G
\]
$a_c:\, H\times V^\ast \rightarrow V^\ast$  
denotes the action associated to the contragredient 
representation of $H$ in $V^\ast$ with respect to $\sigma$. 

\item  We define 
in the SHCP context the assignment:
\begin{align}
\label{eq::assbundleSHCP}
U\mapsto \cA_{SHCP}(U)
\end{align}
where:
\begin{align}
\label{eq::sheafSHCP}
\cA_{SHCP}(U) & := \left\{f\in \cA(U)\, |\, 
 \left\{\begin{array}{ll}
( r_h^\ast\otimes 1) f =(1\otimes {\widetilde{\sigma}(h)^{-1})}(f) & 
\forall h\in \red{H} \\
(D^L_X \otimes 1) f = (1\otimes {\rho^\sigma}{(-X})) f & \forall X\in \hk_1
 \end{array}
 \right.
 \right\}
 \end{align}

\item  We define 
in the functor of points context the assignment:
\begin{align}
\label{eq::assbundleFOP}
U\mapsto \cA_{FOP}(U)
\end{align}
where
\begin{align}
\label{eq::sheafFOP}
\cA_{FOP}(U)  \coloneqq \left\{
f: p^{-1}(U) \rightarrow V \otimes_{k}{k}^{1 \mid  1}  \, 
|\right.  \left. \, 
f_T(gh)=\sigma^\prime_T(h)^{-1} f_T(g)
\right\}, 
\end{align}
with $g \in G(T)$, $h \in H(T)$ and
\[
\sigma^\prime : H\rightarrow \GL(V \otimes_{k}k^{1|1}), \qquad
\sigma^\prime=\sigma\otimes 1
\]
where  $p^{-1}(U) \subset G$ is the open subsupermanifold
corresponding to the open set  $\red{p}^{-1}(U)$
and  $T \in \smflds$.\\
Notice that also $\cA_{FOP}(U) \subset \cA(U)$ 
since elements of the  sheaf $\cO_G(\red{p}^{-1}(U))$ identify with morphisms
of supermanifolds $p^{-1}(U)$ $\rightarrow$ 
$k^{1|1}$.
\end{itemize}
\end{definition}

\medskip

We now 
establish the equivalence of
the three notions introduced in the previous definition.

\begin{proposition}
\label{Proposition::equivbundle}
The assignments 
\begin{align}
\label{eq::assbundle}
U\mapsto \cA_{SLG}(U), \quad U\mapsto \cA_{SHCP}(U), \qquad U \mapsto  
 \cA_{FOP}(U)
\end{align}
define super vector bundles on $G/H$ with fiber isomorphic to $V$. 
Moreover we have
\begin{align}
\label{eq::firsteq}
\cA_{FOP} = \cA_{SHCP}= \cA_{SLG} 
\end{align}
\end{proposition}
\begin{proof}
We first show that $\cA_{SHCP}$ is a super vector bundle on 
the quotient $G/H$.
Let us denote $\cA_{SHCP}$ with $\cF$. We need to show that 
$\cF$ is a sheaf of $\cO_{G/H}$--modules, and that
it is locally free.  $\cO_{G/H}$ acts naturally on the first
component of $\cA$, we now want to show that such an action is
well defined on $\cF$, so that $\cF(U)$ is an $\cO_{G/H}(U)$--module
for all open $U$. Indeed, if $\phi\in\cO_{G/H}(U)$ and $f\in\cF(U)$ 
then $(r^\ast_h \otimes 1)(\phi f) = (1 \otimes \sigma(h))^{-1}(\phi f)$, 
and $(D^L_X \otimes 1)(\phi f)=(1 \otimes \rho(-X))(\phi f)$ 
(due to the right $H$ invariance of $\phi$). 
\\
Moreover it is clear that $\cF$ is a sheaf since, for each open $U\subseteq G/H$, $\cF(U)$ is a sub$\cO_{G/H}(U)$--module of $\cO_G(U)\otimes V$. Using the fact $U \mapsto \cO_G(U)\otimes V$ is a sheaf over $G/H$ and the fact that right $H$--invariance is a local property, it follows that $\cF$ is a sheaf over $G/H$.\\
In order to prove the local triviality of the sheaf $\cF$, we will use the existence of local sections for  $p\colon G\to G/H$. In Ch. 8 in \cite{ccf}
we have the local isomorphism:
\[
\gamma\colon W\times H\to p^{-1}(W)
\]
so that we can define a section:
\begin{align*}
s\colon W & \to p^{-1}(W)
\end{align*}
such that $s^\ast(f)= (1\otimes i_e^*)\gamma^\ast(f)$. 
Notice that $s$ can also be described as 
$\gamma \circ (1\times i_{\set{e}})$ where 
$i_{\set{e}}\colon \set{e}\to H$ is the embedding of the 
topological point $e$ into $H$. \\
Suppose hence that a neighborhood $U$ of $1$ admitting a 
local section $s$ has been fixed. Define the following two maps
\begin{align*}
\eta \colon \cF(U) & \to  \cO{(U)}_{G/H}\otimes V\\
F & \mapsto f_F\coloneqq (s^{\ast}\otimes 1_V)(F)
\end{align*}

and
\begin{align*}
\zeta \colon \cO{(U)}_{G/H}\otimes V & \to  \cF(U)\\
f & \mapsto F_f \coloneqq (\gamma^{\ast}\otimes 1_V)(1_U\otimes a^{\ast}_c)f
\end{align*}

It is easy to check that $\eta$ and $\zeta$ are one the inverse of the other.

\medskip

We now go to the equalities: $A_{SHCP}=A_{SLG}=A_{FOP}$.
The equality $ A_{SHCP}=A_{SLG} $ is proved in \cite{bcc}.
In order to  prove $A_{FOP} =A_{SLG}$, it is enough to 
notice that condition \eqref{eq::sheafFOP} is equivalent to the 
commutativity of the following diagram
\[
\xymatrix{
G\times H \ar^{\mu_{G,H}}[r]\ar@{=}[d] & G \ar^f[r] & V\otimes k^{1|1}\ar@{=}[d]\\
G\times H \ar^{c({f}\times 1_H)}[r] & \qquad H\times (V\otimes k^{1|1}) \qquad
\ar^{{\sigma^\prime}^{-1}}[r] & V\otimes k^{1|1}\\
}
\]
where \(c\colon V\times H \to H\times V\) is the commutation morphism.
\end{proof}

\subsection{Principal Super Bundles}

If 
$E$ and $M$ are smooth manifolds and $G$ is a Lie group, we say 
that $\pi: E \lra M$ is a \textit{$G$-principal bundle}  with total space
$E$ and base $M$, if $G$ acts freely from the right on $E$, trivially on $M$ 
and it is locally trivial, i.e. there exists an open cover $\{U_i\}_{i\in I}$ of 
$M$ and diffeomorphisms 
$$
\sigma_i:\pi^{-1}(U_i) \lra U_i \times G, \qquad
\sigma_i(u)=(\pi(u),h)
$$
such that
$$
\sigma_i(ug)=(\pi(u),hg), \quad g, h \in G. 
$$
$M$ can thus be identified with the orbit space $E/G$.

We want to give the super analogue of this definition in the
different languages we employed in the previous section.

\medskip
Let $E=(\red{E}, \cO_E)$ and $M=(\red{M},\cO_M)$ be supermanifolds 
and $G$ a SLG acting on $E$ from the right.
Assume we have a surjective submersion
$\pi:E \lra M$.
Assume we have an open cover $\{\red{U_i}\}$ of $\red{M}$ and
diffeomorphisms
$\sigma_i:\pi^{-1}(U_i) \lra U_i \times G$
making the following diagram commute:
\beq \label{comm-diagr}
\xymatrix{
\pi^{-1}(U_i) \ar[rd]^\pi \ar[r]^{\sigma_i} &  U_i \times G 
\ar[d]^{\mathrm{pr}_1} \\
& U_i
}
\eeq
where
now $U_i=(\red{U_i}, \cO_M|_{\red{U_i}})$ and 
$\pi^{-1}(U_i)=(\red{\pi}^{-1}(\red{U_i}),\cO_E|_{\red{\pi}^{-1}(\red{U_i})})$ 
are supermanifolds.

\begin{proposition} \label{eq-cond}
Let the notation be as above. 
Let $a:E \times G \lra E$ be the right action of $G$ on $E$.
The following three conditions are equivalent:
\begin{enumerate}
\item (Sheaf theoretic approach)
\beq \label{dual-eq}
a^* \cdot \sigma_i^*=(\sigma_i^* \otimes 1)(1 \otimes \mu^*)
\eeq
where $\mu$ is the multiplication in $G$.
\item (SHCP approach)
\beq \label{shcp-eq}
i) \, \wt{\sigma_i} \cdot \wt{a}=(1\times \wt{\mu})(\wt{\sigma_i} \times 1), \qquad ii) \,
\rho_a \circ \sigma_i^*=(\sigma_i^* \otimes 1)(1 \otimes \rho_\mu)
\eeq
where:
\begin{itemize}
\item $\red{a}:E \times \red{G} \lra \red{G}$ is the action
of the ordinary Lie group $\red{G}$ on the supermanifold $E$ (similar
meaning for $\wt{\sigma_i}$ and $\wt{\mu_i}$).
\item $\rho_a:\fg \lra \rVec(E)^\op$, $\rho_a(X)=(1 \times X_e)a^*$,
$\fg=\Lie(G)$
\item $\rho_\mu: \fg \lra \rVec(G)^\op$, $\rho_\mu(X)=(1 \times X_e)\mu^*$
\end{itemize}
(see \cite{ccf} Ch. 8 for more details on the SHCP language).
\item (Functor of points approach):
\beq \label{fopts}
(\sigma_i)_T(ug)=(\pi_T(u),hg), \qquad g,h \in G(T)
\eeq
where $T \in \smflds$.
\end{enumerate}
\end{proposition}

\begin{proof} We first show that (1) is equivalent to (3).
Let us choose, without loss of generality, a covering $\{U_i\}_{i \in I}$ by
superdomains (see \cite{ccf} Sec. 3.2). By the Global Chart theorem
(see \cite{ccf} Thm. 4.2.5), we have that we can express in the
functor of points notation the
 diagram (\ref{comm-diagr}) as:
$$
(\sigma_i)_{T}(u)=(\pi_T(u), h), \qquad u \in \pi^{-1}(\red{U_i})(T),\,
h \in G(T)
$$
for $T \in \smflds$. So the condition (3) of our proposition makes
sense as it is written, recalling that $ug \in \pi^{-1}(U_i)(T)$
and $gh \in G(T)$ are defined as:
$$
ug=m \cdot u \otimes g \cdot a^*, \qquad gh=m \cdot g \otimes h \cdot \mu^*
$$
$m$ being the multiplication in $\cO(T)$ (see \cite{ccf} Ch. 10). 
Because of the equivalence between the functor of points morphisms
and morphisms of supermanifolds, we can write the following diagram:
\beq \label{comm-diagr2}
\xymatrix{
\pi^{-1}(U_i) \times G \ar[r]^{\sigma_i \times 1} \ar[d]_a & 
U_i \times G \times G \ar[d]^{1 \times \mu} \\
  \pi^{-1}(U_i)       \ar[r]_{\sigma_i}       & U_i \times G }
\eeq
which on the sheaves proves immediately the equivalence 
between (1) and (3). We now show
that (1) and (2) are equivalent.
By Prop. 8.3.2 and 8.3.3 the action in the
language of SHCP grants the existence of $\wt{a}$ and $\rho_a$.
The diagram (\ref{comm-diagr2}) expressed in the language of
SHCP's gives the equivalence between the conditions (1) and (2).
Condition (i) is immediate from diagram (\ref{comm-diagr2}), while (ii)
comes directly from the definitions of $\rho_a$ and $\rho_\mu$.

\end{proof}

\begin{definition}
We say that a SLG $G$ acts \textit{freely} on the right on a supermanifold $E$
if we have an action $a:E \times G \lra E$ and
the group $G(T)$ acts freely on the right on the set $E(T)$ for
all supermanifolds $T$, via the natural transformation
$a_T:E(T) \times G(T) \lra E(T)$.
\end{definition}

We are ready to give the definition of {\sl principal super bundle}.

\begin{definition}
Let $E$ and $M$ be supermanifolds and $G$ a SLG.
We say that a surjective submersion
$\pi:E \lra M$ is a 
\textit{principal super bundle}  with total space
$E$ and base $M$, if $G$ acts freely from the right on $E$, trivially on $M$,
and we have an open cover $\{\red{U_i}\}$ of $\red{M}$ and
isomorphisms
$\sigma_i:\pi^{-1}(U_i) \lra U_i \times G$
making the diagram (\ref{comm-diagr}) commute and such that
the three equivalent conditions of Prop. \ref{eq-cond} are
satisfied.
\end{definition}

\section{Applications}

In this section we examine some important applications of the theory
of associated and principal super bundles described above.

\medskip
Let $\fg$ be a complex contragredient Lie superalgebra, namely $\fg$
is one of:
$$
A(m,n), \, m \neq n, \quad B(m,n), \quad C(n), \quad D(m,n) \quad
D(2,1;\al), \, F(4), \, G(3)
$$
Let $\fh$ be a Cartan subalgebra (recall $\fh \subset \fg_0$).
Let $G$ be a complex simply connected analytic supergroup with $\fg=\Lie(G)$;
we call such a $G$ \textit{simple}.
Let $B$ a Borel subsupergroup of $G$,
namely the subsupergroup associated with 
a fixed Borel subalgebra of $\fg$ (i.e. we fix a positive system)
and let $T$ be the torus associated
with $\fh$. Let $P$ be a subsupergroup containing $B$.
We call such supergroups \textit{parabolic subsupergroups}. 
Let $\chi:P \lra \C^\times$ be a character of $P$, $\fp=\Lie(P)$. 
Hence by Prop. \ref{Proposition::equivbundle} we can
define a line bundle on $G/P$ and its sheaf is:
\beq\label{lineb}
\begin{array}{rl}
\cL^\chi(U)=&\left\{f:  p^{-1}(U)\rightarrow 
{\C}^{1 \mid  1} \, 
|\right.   \, \\  \\
&f_T(gb)=\chi_T(b)^{-1} f_T(g), \, g \in G(T), b \in P(T)
\left. 
\right\}=\\ \\
=&\left\{\right.
f\in \cO_G(p^{-1}(U))\, |\,  
 r_h^\ast f =\chi_0(h)^{-1} (f),\, 
\, \forall h\in \fp_0 , \\ \\
&  D^L_X  f = \lambda{(-X}) f, \,\,  \forall X\in \fp_1
 \left.\right\} 
\end{array}
\eeq
in the language of functor of points and SHCP respectively ($\lambda=d\chi_e$).

\subsection{The Borel-Weil-Bott Theorem}

We want to realize the irreducible finite dimensional holomorphic
representations of 
$G$ in the vector superspace of holomorphic sections of a certain
super line bundle on $G/B$ and prove the
super version of the Borel-Weil-Bott theorem, which
was first established in \cite{penkov}, with a different
approach. Our treatment is similar
to the one in 
\cite{cfv}, where, however, the main accent is on infinite 
dimensional representations
of the real supergroup underlying $G$. 

Let $\fg = \fn^- \oplus \fh \oplus\fn^+$, where $\fn^\pm$ are
nilpotent subalgebras, $\fb:=\fh \oplus \fn^+=\Lie(B)$ the corresponding
borel subsuperalgebra
and $\fb^-:=\fh\oplus \fn^-$ is the borel subsuperalgebra opposite to $\fb$.
Let $N^\pm$, $T$ be the subSLG in $G$ corresponding to $\fn^\pm$,
$\fh$ respectively. Fix $\chi:T \lra \C^\times$ a character of the torus $T$
and extend it trivially to the whole $B$. Let $\lambda \in \fh^*$,
$\chi=exp(\lambda)$ (here the exponential offers no difficulties
since $T$ is even). Since the character $\chi$ is determined by $\lambda$, we
shall denote the line bundle $\cL^\chi$ also by $\cL^\lambda$ and the character
$\chi$ by $\chi_\lambda$.

\medskip
The topological space $\wt{N^-}\wt{T}\wt{N^+}$ is open in 
$\wt{G}$ and defines an open subsupermanifold of $G$ called the
\textit{big cell}, that we denote with $\Gamma$. It is not
difficult to see that the morphism, given in the functor of points
notation as:
$$
N^- \times T \times N^+ \lra G, \qquad (n^-,h,n^+) \mapsto n^-hn^+
$$
is an analytic diffeomorphism onto $\Gamma$.
Similarly we have an analytic diffeomorphism of $N^- \times B$
onto $\Gamma$. Hence we can identify quite naturally
the quotient $\Gamma/B$ with the subsupergroup $N^-$. 
We now state a lemma, which is an easy consequence of
what we have detailed above.

\begin{lemma}
Let the notation be as above. 
Then we have an isomorphism
identifying the sections of $\cL^\lambda$ on $\Gamma/B$ with the global
holomorphic sections on $N^-$:
\beq\label{id-tilde}
\cL^\lambda(\Gamma/B) \cong \cO(N^-)
\eeq
\end{lemma}

Let $t_\al$ denote the global exponential coordinates on $N^-$, 
$\al \in \Delta^-$, the negative roots
(see \cite{gw}, \cite{fg, fg2, fg3}). Formally, using
the functor of points notation, we have:
$$
t_\al(\exp(
x_\al X_\al))=x_\al, \quad\al \in \Delta^-, \quad x_\al \in \cO(S),
\, S \in \smflds
$$
where $X_\al$ is the root vector of $\al$ in a fixed Chevalley
basis for $\fg$ (see \cite{fg}).
 
Let $\cP\subset \cO(N^-)$ be the polynomials
in the $t_\al$'s. Denote with $\widehat{\cP}$ the corresponding elements
in $\cL^\lambda(\Gamma/B)$ according to the
identification (\ref{id-tilde}).

\medskip
As in the ordinary setting, $T$ acts on the
big cell $\Gamma$ by left translation:
$$
a \cdot n^-b:= an^-b, \qquad a \in T(S), \, n^-\in N(S), \, b \in B(S)
$$
for $S \in \smflds$. This action is well defined since
$$
an^-b=an^-a^{-1}ab, 
$$
and one can check that $an^-a^{-1} \in N(S)$, $ab \in B(S)$
(see \cite{cfv} for more details and also \cite{fg} for
the explicit realization of these subgroups, which make
the statements obvious). Then, we can define a representation
of $T$ in $\cL^\lambda(\Gamma/B)$ in the same fashion 
as Def. \ref{contra-rep}, here using the functor of points notation:
$$
(a \cdot f)(g)=f(a^{-1}g)
$$
We can explicitly compute the action of the
maximal torus $T$ on $\wh{\cP}$. 

\begin{proposition} \label{Proposition-Taction} 
The torus $T$ acts on $\widehat{\cP} \subset \cL^\lambda(\Gamma/B)$ 
and we have that:
$$
a \cdot \widehat{t^{r_{\al_1}}_{\al_1} \dots t^{r_{\al_s}}_{\al_s}} = 
\chi_{\lambda+\sum r_{\al_i}\al_i}(a) \,  
t^{r_{\al_1}}_{\al_1} \dots t^{r_{\al_s}}_{\al_s}
$$
Hence $\widehat{\cP}$ decomposes into the sum of eigenspaces 
$\widehat{\cP_d}$, where $d$ ranges in $D^+$
the semigroup in $\fh^*$ generated by the positive roots.
$$
\wh{\cP}=\oplus_{d \in D^+}\wh{\cP_d}, \qquad 
\wh{\cP_d}=\oplus_{\sum r_{\al_i}\al_i=d} \, \C \cdot 
\wh{t^{r_{\al_1}}_{\al_1} \dots t^{r_{\al_s}}_{\al_s}}
$$
A similar decomposition holds also for $\cP\subset \cO(N^-)$.
\end{proposition}

\begin{proof}
Let us do this just for $t_\al$, the general calculation being the same.
$$
\begin{array}{rl}
(a \cdot \widehat{t_\al})(n^-b)&=
\widehat{t_\al}(a^{-1}n^-b)= \wh{t_\al}(a^{-1}na \cdot a^{-1}b)=
\chi_\lambda(a)\wh{t_\al}(a^{-1}na)=\\ \\
&=\chi_\lambda(a)t_\al(a^{-1}na)=\chi_\lambda(a) \chi_\al(a)
t_\al(n)=\chi_{\lambda+ \al}(a)t_\al(n).
\end{array}
$$
where, as in the ordinary setting, one can easily show that
$t_\al(a^{-1}na)=\chi_\al(a)t_\al(n)$.

Since $\widehat{t_\al}$ is determined by its restriction to
$N^-$, under the identification (\ref{id-tilde}), we have obtained:
$$
a \cdot \widehat{t_\al}=\chi_{\lambda+ \al}(a)\wh{t_\al}, \qquad
a \cdot {t_\al}=\chi_{\lambda+ \al}(a)t_\al 
$$ 
from which our statement follows.
\end{proof}

\begin{definition}
There are two well defined actions of $\fg$, hence 
of $\cU(\fg)$, on $\cO(\Gamma)$ the global (holomorphic)
sections on the big cell $\Gamma$, that read as follows:
$$
\begin{array}{rl}
\ell(X)f&=(-X \otimes 1)\mu^*(f) , \qquad X \in \fg \\ \\
\partial(X)f&=(1 \otimes X)\mu^*(f)
\end{array}
$$
\end{definition}

Notice that the actions $\ell$ and $\partial$
are well defined also on $\cL^\lambda(\Gamma/B)$ and they commute with 
each other,
furthermore, being algebraic, they leave $\wh{\cP}$ invariant (see
\cite{cfv} for more details).

\begin{theorem} \label{Theorem5} 
\begin{enumerate}
\item
There is a $\cU(\fg)$-pairing between $\wh{\cP}$ and $\cU(\fg)$:
$$
\langle,\rangle:\wh{\cP} \times \cU(\fg) \lra \C, \qquad
\langle f,u \rangle:=(-1)^{|u||f|}(\partial(u)f)(1_G)
$$
\item The above pairing gives a non singular pairing
between $\wh{\cP}$ and the Verma module $V_\lambda=\cU(\fg)/\cM_\lambda$.

\item The submodule $\cI_\lambda$ of  $\wh{\cP}$ generated 
by  the constant function $1$  is irreducible and it is the unique
irreducible submodule of  $\wh{\cP}$
of lowest weight $-\lambda$.
\end{enumerate}
\end{theorem}

\begin{proof} (Sketch). The fact that we have a $\cU(\fg)$-pairing
between $\cU(\fg)$ and $\wh{\cP}$ 
amounts to a tedious check. Then one can show it factors to a non
singular
pairing between $V_\lambda=\cU(\fg)/\cM_\lambda$ and $\wh{\cP}$, by
showing $\langle f, u\rangle=0$ for  $u \in \cM_\lambda$.
This establishes a duality between these two $\cU(\fg)$ modules,
which is actually an isomorphism, since they have the
same weight spaces by \ref{Proposition-Taction}. Hence,
since $V_\lambda$ has a unique irreducible quotient, by duality $\wh{\cP}$ will
have a unique irreducible submodule
$\cI_\lambda$ of lowest weight $-\lambda$ (see \cite{cfv} for more details).
\end{proof}

We now define the following action of $G$ on $\cL^\lambda(G/B)$:
$$
(g \cdot f)=f(g^{-1}x)
$$
using the functor of points notation, or equivalently in the
language of SHCP:
\beq\label{action-shcp}
\begin{cases}
(g \cdot f) =l_{g^{-1}}^*f & g\in \red{G_r}\\
X.f  = D^R_{\overline{X}} f & X\in \fg
\end{cases}
\eeq
(where, $\overline{X}$ is the antipode
of $X \in \cU(\fg)$).

\begin{theorem} \label{bwb-thm}
(The Borel-Weil-Bott Theorem).
Let $G$ be a simple simply connected complex supergroup,
$B$ a borel subsupergroup, $\fh$ a CSA of $\fg=\Lie(G)$.
Then all irreducible
finite dimensional representations of $G$
are realized as $\cI_\lambda \subset \cL^\lambda(G/B)$, $\lambda \in \fh^*$ 
dominant integral for the numerical marks $a_i$ as in \cite{kac}.
\end{theorem}

\begin{proof}
We first need to show that  $\cI_\lambda \subset \cL^\lambda(G/B)$ is stable
under the $G$ action. 
We will do this by using the SHCP approach: $\cI_{0,\lambda}$
is stable under the $\red{G}$ action, by the classical theory; $\cI_\lambda$ is
stable under the $\cU(\fg)$ action (see Thm \ref{Theorem5}) and
such action is the differential of the $G$ action (immediate from
(\ref{action-shcp})). Hence given   $\lambda \in \fh^*$, $\cI_\lambda$ is 
a $G$ representation and it
is finite dimensional because it is dual
to a Verma module $V_\lambda$, where $\lambda$ is dominant integral and 
the numerical marks verify the conditions
in \cite{kac}\footnote{These conditions are necessary because the
Weyl group does not act transitively on the set of borel subsuperalgebras.}.
If $W$ is any finite dimensional $G$ representation, by taking
its differential, we obtain a finite
dimensional $\fg$ representation, corresponding to
a dominant integral $\lambda \in \fh^*$, with numerical marks
$a_i$ as in \cite{kac}. Then we can build $\cI_\lambda$, which is
a weight module as $W$ with same weight spaces and weights, so
they are isomorphic. 
\end{proof}

\subsection{Projective embeddings of homogeneous spaces}

In ordinary geometry ample line bundles on varieties
give projective embeddings, and in particular, the complex analytic manifold
$\red{G}/\red{P}$ always admits projective embeddings. 
It is well known that this is not
the case in projective supergeometry, namely there are complex
analytic supermanifolds obtained as
 $G/P$, $P\supset B$, which do not admit any projective embedding. The easiest
example is $\mathrm{Gr}(1|1;2|2)$ the Grassmannian supermanifold
of $1|1$ spaces into $\C^{2|2}$. This is obtained as the quotient
$\rSL(2|2)/P$, for a suitable parabolic $P$ (see \cite{ccf} Sec. 10.5
for more details).

Nevertheless, once this anomaly is
resolved, we can extend the theory of projective embeddings
to supergeometry.

\medskip
In order to do this, let us define:
$$  
\cO(G/P)_n  \; := \; \cL^{\chi^n}(G/P)=\left\{f:  G\rightarrow 
{\C}^{1 \mid  1} \, 
|\right.   \,   f_T(gb)=\chi_T(b)^{-n} f_T(g)
\left. \right\}  
$$
This is the superalgebra of global sections of the line
bundle $\cL^{\chi^n}$ which is associated with the character $\chi^n$.
Let us also define:
$$   
\cO(G/P)  \; := \;  \bigoplus_{n \geq 0}
\, \cO(G/P)_n \; \subseteq \;  \cO(G)  \quad  
$$
This is a $\N$-graded algebra; in fact we can easily
verify that if $f \in \cO(G/P)_n$ and $g \in \cO(G/P)_m$
their product $fg\in  \cO(G/P)_{n+m}$, $m,n \in \N$. 
We say that $\cL$ is \textit{very ample} if $\cO(G/P)$
is generated in degree 1, i.e. there exist
 $ f_0 $,  $ f_1 $,  $ \dots $,  
$ f_m$, $\phi_1$, $\dots$, $\phi_n \in \cO(G/P)_1 \, $ generating $\cO(G/P)$
as commutative $\Z$-graded superalgebra. 

\begin{proposition}
Let $\cL$ be a very ample line bundle on $G/P$ as above.
Then $\cL$ gives a projective embedding of 
the complex analytic supervariety $G/P$ into $\bP^{m|n}$.
\end{proposition}

\begin{proof}
See \cite{ccf} Ch. 9, the proof
is the same as in the ordinary setting. 
\end{proof} 

The superalgebra $\cO(G/P)$ is called the \textit{coordinate superalgebra}
of $G/P$ with respect to the given projective embedding.
We also notice that this provides $G/P$ with a structure of algebraic
supervariety, besides the one of complex analytic supermanifold.

\medskip 
We want to characterize such an embedding in purely algebraic terms.
This is especially fruitful if we want to discuss quantum
deformations (see \cite{fg1, fi1}).

\begin{proposition} \label{t}
Let the notation be as above.
Let $G/P$ be embedded into some projective space via some
very ample line bundle. Then there exists a $t \in \cO(G)$ such
that
$$  
\Delta_\pi(t) \, := \, \big((\text{\it id} \otimes \pi)
\circ \Delta \big)(t) \, = \, t \otimes \pi(t)
$$
$$  
\pi\big(t^m\big) \not= \pi\big(t^n\big) \quad \forall \;\; m \not= n 
\in \N
$$
$$  
\cO(G/P)_n  \; = \;  \Big\{\, f \in \cO(G) \,\;\Big|\;
(id \otimes \pi)\Delta(f) = f \otimes \pi\big(t^n\big) \Big\}    
$$
where $\pi: \cO(G) \lra \cO(P):=\cO(G)/I_P$.
Furthermore,
$\cO(G/P)=\bigoplus_{n \in \N} \; \cO(G/P)_n $  is generated in degree 1.

Vice-versa, if such $t$ exists, it gives a projective embedding of $G/P$.
\end{proposition}

\begin{proof} The proof is the same as in the ordinary setting,
however for completeness and given the peculiarity of projective
embeddings of supermanifolds, we include it. 
Let $\Lambda=S(\chi) \in \cO(P)$,
where $S$ denotes the antipode in $\cO(P)$.

\medskip
By assumption there exists a non-zero global section of the line
bundle on  $ G\big/H \, $,  i.e. a section  $ \, t \in \cO(G/P)_1
\setminus \{0\} \, $  on  $ G \, $
and $\ep(t) \neq 0$.  
Up to dividing out by  $ \epsilon(t) $,  we can assume that  $ \, \pi(t)
= \Lambda \, $. The condition defining $\cO(G/P)_n=\cL^{\chi^n}(G/P)$
can be rephrased as:
$$
f \circ \mu_{G,P}=f \otimes \chi^{-n} , \qquad f \in \cO(G), \,
\mu_{G,P}=\mu|_{G \times P}
$$
Take the sheaf theoretic picture, then $\mu_{G,P}^*=(\id \otimes \pi)\Delta$
and, as we noticed, we can choose $t$ so that $\pi(t)=S(\chi)$.
The result follows immediately.
\end{proof}

 The element $t \in \cO(G)$ essentially defines the line bundle
giving the projective embedding of $G/P$ and we call it a
\textit{classical section}

\medskip
We now want to show that the associated super bundle providing the
projective embedding of $G/P$ actually enables us to construct explicitly
the principal bundle structure for the projection morphism $\pi:G \lra G/P$
(here the total space is $E=G$, while the supergroup acting is $P$).

\begin{theorem} \label{mainthm-emb}
Let $G$ be a simple complex analytic supergroup. Then the
projection morphism $G \lra G/P$ is a principal bundle. 
If $G/P$ admits a projective embedding via a classical section $t$
then a local trivialization of the principal bundle is given
by the affine open subset corresponding to the invertibility
of the family $\{t^{(2)}\}$ defined as:
$\Delta(t)=\sum t^{(1)} \otimes t^{(2)}$.
\end{theorem}

\begin{proof}
For the first assertion, notice that $P$ acts freely on $G$ and
trivially on $G/P$, a local trivialization with the properties
of Prop. \ref{eq-cond} is obtained by the very construction
of quotients (see \cite{ccf} Ch. 8). As for the second assertion,
it is the same as in the ordinary setting, but we briefly recap it.
By the Borel-Weil-Bott Theorem \ref{bwb-thm}, we have that
$\cO(G/P)_1$ is an irreducible representation of $G$. In the
functor of points notation we can write:
$$
(g \cdot t)(x)=t(g^{-1}x)=\Delta(t)(g^{-1} \otimes x)=
\sum t^{(1)}(g^{-1}) \otimes t^{(2)}(x).
$$
It is then clear that the affine open sets defined as
complement of $\{\red{t^{(2)}}=0\}$ will cover $G$, otherwise
we would have a common zero for global sections of a line
bundle giving a projective embedding. 
\end{proof}

\begin{remark}
Everything we say for the complex analytic supergroup $G$ in this
section can be generalized to the complex algebraic category. 
The Borel-Weil-Bott statement is true and the proof is the same,
provided we consider the objects in the correct category.
As for the principal bundles: the existence of a local trivialization
for the bundle $G \lra G/P$ is not granted in general for the
algebraic category even in the ordinary setting, it is however true
for the simple supergroups that we are considering. We shall not
pursue this question further in the present work leaving a full
discussion of the quotients of simple supergroups in a forthcoming paper.
\end{remark}

\end{document}